\documentclass[11pt]{amsart}
\usepackage{amsmath,amssymb,amsthm,amscd,amsfonts}
\usepackage{enumerate}
\usepackage{color}
\usepackage{comment}

\newtheorem{innercustomgeneric}{\customgenericname}
\providecommand{\customgenericname}{}
\newcommand{\newcustomtheorem}[2]{%
  \newenvironment{#1}[1]
  {%
   \renewcommand\customgenericname{#2}%
   \renewcommand\theinnercustomgeneric{##1}%
   \innercustomgeneric
  }
  {\endinnercustomgeneric}
}

\newcustomtheorem{customthm}{Theorem}
\newcustomtheorem{customlemma}{Lemma}

\theoremstyle{plain} \numberwithin{equation}{section}
\newtheorem{thm}{Theorem}[section]

\newtheorem{cor}[thm]{Corollary}
\newtheorem{conj}{Conjecture}
\newtheorem{lem}[thm]{Lemma}
\newtheorem{prop}[thm]{Proposition}
\newtheorem*{clm*}{Claim}

\theoremstyle{definition}
\newtheorem{defn}[thm]{Definition}

\newtheorem{example}[thm]{Example}
\newtheorem{quest}[conj]{Question} \topmargin-2cm

\newcommand{\diam}{\operatorname{diam}}
\renewcommand{\int}{\operatorname{int}}

\newcommand{\cl}{\operatorname{cl}}
\newcommand{\cll}{\operatorname{cl}_{\mathbb S^2}}
\newcommand{\im}{\operatorname{Im}}

\begin{document}

\title{Peano dimension of fundamental groups}

\author{Gregory R Conner${}^1$}\thanks{$^1$Supported by Simons Foundation Collaboration Grant 646221.}

\author{Curtis Kent${}^2$}\thanks{$^2$Supported by Simons Foundation Collaboration Grant 587001.}

\begin{abstract}
We define the Peano dimension for groups arising as fundamental groups, which generalizes the classical definition of geometric dimension of finitely presented groups.  We conjecture that the Peano dimension of the fundamental group of a aspherical Peano continuum $X$ is equal to the homotopy dimension of $X$.  We prove the conjecture for one-dimensional or planar Peano continua.  This answers a question posed by  Cannon and Conner in 2007 concerning the homotopy dimension of planar sets.

\end{abstract}

\maketitle

\section{Introduction}

The geometric group theory concerns itself with the relationships between algebraic properties of groups and topological and geometric properties of spaces on which they act.  The \emph{geometric dimension} and \emph{cohomological dimension} of a group form a  classical and well-used pair of notions relating them.  These two dimensions coincide for finitely presented groups (except in dimension 2 where their equality is the Eilenberg-Ganea conjecture) and form a basis for many standard constructions and arguments in the area.  The reason for their utility is that cohomological dimension has a purely algebraic definition while geometric dimension is entirely   topological. This allows one to move back and forth between the worlds of algebra and topology -- with the standard Galois correspondence from covering space theory as a major tool.

Unfortunately for fundamental groups of non-locally contractible spaces and other uncountable groups both of these notions are very difficult to compute.  In \cite{cc4}, Cannon and Conner consider the notion of \emph{homotopy dimension} and discern which subspaces of the plane can be homotoped to be one-dimensional.  Their constructions are topological in nature.  The key advantages of homotopy dimension are that it is defined in a much broader context and it agrees with geometric dimension for finitely presented groups.

The \emph{geometric dimension} of a finitely presented group $G$ is the minimal dimension of a space with fundamental group $G$  and contractible universal cover.
Such spaces for fundamental groups of locally complicated spaces are unwieldy. Alternatively, one could replace the requirement of having a contractible universal cover with the property that all higher homotopy groups are trivial.  This property is equivalent when considering CW-complexes but allows for potentially nicer spaces for many uncountable groups, e.g. fundamental groups of spaces with arbitrarily small essential loops.  Thus we will define a dimension for groups arising as fundamental groups of spaces with nontrivial local algebraic topology as follows.

\begin{defn}
	The \emph{Peano dimension} of a group $G$ is the minimum covering dimension among aspherical Peano continua with fundamental group $G$, if any such Peano continua exist.  The Peano dimension of $G$ is undefined if there does not exits an aspherical Peano continuum with fundamental group $G$.
\end{defn}


The geometric dimension of a finitely presented group is equal to its cohomological dimension whenever the cohomological dimension is not 2 \cite{EilenbergGanea57,Stallings68,Swan69}.  Thus the fundamental group of the Hawaiian earring has geometric dimension at least 2, cohomological dimension at least 2, Peano dimension 1, and homotopy dimension 1.  Thus homotopy dimension and the Peano dimension create a better matching of spaces to fundamental groups than cohomological dimension and geometric dimension for many locally complicated spaces.


The \emph{homotopy dimension} of a space is the minimum covering dimension among all Hausdorff spaces homotopy equivalent to it.  The classical notion of homotopy dimension is a natural mixture of homotopy and topology and has been studied by several authors, see \cite{cc4, Liem81, Wall65}.  It is an exercise to show that the geometric dimension of a finitely presented group $G$ is the homotopy dimension of any $K(G,1)$ space.



We make the following conjecture relating the Peano dimension of fundamental groups to the homotopy dimension of an aspherical Peano continua with the prescribed fundamental group, which we prove for planar and one-dimensional Peano continua.



\begin{conj}\label{conjecture}
If $X$ is an aspherical Peano continuum with fundamental group $G$, then the Peano dimension of $G$ is the homotopy dimension of $X$.
\end{conj}

\noindent Note the conjecture simply states that any two aspherical Peano continua with isomorphic fundamental groups have the same homotopy dimension.  Thus the homotopy dimension would be an invariant of the fundamental group in this setting.

Both planar spaces and one-dimensional spaces are aspherical \cite{ccz}.  Thus to prove Conjecture \ref{conjecture} for planar or one-dimensional Peano continua, we need only prove the following:



\begin{thm}\label{main}\hspace{3in}
     \begin{enumerate}
        \item A planar Peano continuum with homotopy dimension one is homotopy equivalent to a one-dimensional planar Peano continuum.

     	\item In the class comprised of the union of one-dimensional Peano continua and planar Peano continua, the fundamental group determines the homotopy dimension.
     \end{enumerate}
\end{thm}


Every contractible space has homotopy dimension  zero, and every planar continuum has homotopy dimension at most two. The space obtained by filling one removed square of the Sierpinski carpet gives an example of a planar continuum with homotopy dimension exactly two \cite{ccz}.  In fact, there exist uncountably many non-homotopy equivalent planar continua with homotopy dimension exactly two \cite{Kent18}.  Theorem \ref{main} gives us the following corollary, which answers Question 1.6 of \cite{cc4}.

\begin{cor}Let $X$ be a planar Peano continuum.  Then $X$ has homotopy dimension one if and only if the fundamental group of $X$ is isomorphic to the fundamental group of a non-contractible one-dimensional Peano continuum.
\end{cor}

We will also show that the properties of being homotopically a Peano continuum, homotopically one-dimensional, or homotopically planar are rigid in the following sense.

\begin{thm}\label{main2}
Let $X$ be a topological space.  If $X$ is homotopy equivalent to spaces $Y_1$ and $Y_2$ where $Y_1$ is one-dimensional and $Y_2$ is a Peano continuum, then $X$ is homotopy equivalent to a one-dimensional Peano continuum.  If, in addition, $X$ is homotopy equivalent to a planar set, then $\pi_1(X,x_0)$ is isomorphic to the fundamental group of a one-dimensional planar Peano continuum.
\end{thm}

Cannon and Conner  showed that the fundamental group of a planar Peano continuum always embeds into the fundamental group of a one-dimensional Peano continuum \cite{cc4}.  Theorem \ref{main} shows that Cannon and Conner's embedding can never be surjective, if the planar continuum has homotopy dimension two.  Thus there exists a planar continuum with homotopy dimension two whose fundamental group embeds into the fundamental group of a one-dimensional space.

While the fundamental group determines the topology of the set of points at which a planar or one-dimensional Peano continuum is not  locally simply connected \cite{ce, ConnerKent19},  Theorem \ref{main} shows that there exist planar Peano continua with non-isomorphic fundamental groups, and the set of points at which they are not locally simply connected are homeomorphic (see Proposition \ref{not perfect}).

In a subsequent paper \cite{Kentpreprint2} by the second author, it is proved that in the union of the set of planar Peano continua and the set of one-dimensional Peano continua the fundamental group is a perfect invariant of the homotopy type.  Note that this subsequent result, while related, does not imply Theorem \ref{main}, since Theorem \ref{main} requires one to consider the situation when a planar Peano continuum has the fundamental group of an arbitrarily one-dimensional Hausdorff space, which is not covered in \cite{Kentpreprint2}.  The work in \cite{Kentpreprint2} depends on an extensive and technical study of homotopies of planar sets.  Here, on the other hand,  we give a shorter and independent proof of Theorem \ref{main} that illustrates how the classical Phragm\'{e}n-Brouwer properties, \cite[p. 47]{Wilder49}, can be used to reduce many questions about planar continua to the appropriate questions for one-dimensional continua.  Theorem \ref{main} also holds when $X$ is allowed to be the complement of a discrete set in $\mathbb S^2$ (see Theorem \ref{almostmain}), which does not follow from the results in \cite{Kentpreprint2}.

\section{Codiscrete subsets of $\mathbb S^2$}

Cannon and Conner showed that every planar Peano continuum is homotopy equivalent to a nice subset of the $2$-sphere which can be easier to manipulate than a general planar continuum.  A subset of the $2$-sphere, $\mathbb S^2$, is \emph{codiscrete} if it is the complement of a discrete set.

\begin{thm}[{\cite[Theorem 1.2]{cc4}}]\label{classification}
Every Peano continuum $M$ in the 2-sphere $\mathbb S^2$ is homotopy equivalent to a codiscrete subset $X$ of $\mathbb S^2$. Conversely, every codiscrete subset $X$ of $\mathbb S^2$ is homotopy equivalent to a Peano continuum $M$ in $\mathbb S^2$.
\end{thm}

We will use Theorem \ref{classification} and Theorem \ref{almostmain} to prove Theorem \ref{main}.

\begin{defn}

If $X$ is a codiscrete subset of $\mathbb S^2$, let $D(X)$ be the discrete complement of $X$ in $\mathbb S^2$ and $B(X)$ be the set of accumulation points of $D(X)$.  It is immediate that $B(X)$ is the set of points at which $X$ is not semilocally simply connected. For any space $Y$, we will use $B(Y)$ to denote the set of points at which $B(Y)$ is not semilocally simply connected.
\end{defn}

Cannon and Conner were able to prove the following topological characterization of codiscrete subsets of $\mathbb S^2$ with homotopy dimension at most one \cite{cc4}.

\begin{thm}[Cannon \& Conner]\label{chara}

Suppose that $X$ is a codiscrete subset of the two-sphere $\mathbb S^2$. Then $X$ is homotopically at most one-dimensional if
and only if the following two conditions are satisfied.

\begin{enumerate}[(i)]

	\item \label{i}Every component of $\mathbb S^2 \backslash B(X)$ contains a point of $D(X)$.

	\item \label{ii} If $A$ is any closed disk in the two-sphere $\mathbb S^2$, then the components of $A\backslash B(X)$ that do not contain any point of $D(X)$ form a null sequence.
	
\end{enumerate}
	
\end{thm}

When the codiscrete space $X$ satisfies the two conditions of Theorem \ref{chara}, Cannon and Conner actually build a deformation retraction of $X$ onto a one-dimensional subset of $X$ (see \cite[p. 60]{cc4}), which implies the following corollary.

\begin{cor}[Cannon \& Conner]\label{cor: one-dimensional planar}

Suppose that $X$ is a codiscrete subset of the two-sphere $\mathbb S^2$ satisfying both conditions of Theorem \ref{chara}, then $X$ is homotopy equivalent to a one-dimensional planar continuum.
	
\end{cor}

The idea for the proof of Theorem \ref{chara} is to use the \emph{holes} arising from $D(X)$ to push the two-dimensional components of $\mathbb S^2\backslash B(X)$ onto some nice one-dimensional core.  
The problem is that puncturing the Warsaw disc (the closure of the bounded component of complement of the Warsaw circle in the plane) at any finite set of points is insufficient to be able to retract it onto any one-dimensional subset.  Thus condition (\ref{ii}) must be used to guarantee sufficient punctures to build nice retracts.

To prove Theorem \ref{main}, we will use Theorem \ref{continuous} to find a continuous function between the corresponding  codiscrete subsets of the two-sphere which induces an isomorphism of the fundamental groups.  The key step in the proof that is new here is the ability to recognize conditions (\ref{i}) and (\ref{ii}) from the fundamental group.  Section \ref{proof} is dedicated to showing that if a continuous function between a codiscrete sets induces an isomorphism on the fundamental group and factors through a map into a one-dimensional space then the conditions of Theorem \ref{chara} must be satisfied.

\section{Studying codiscrete subsets of $\mathbb S^2$}
The following are standard definitions that we present here to fix notation.

\begin{defn}

Let $B_r^X(x) = \{ y\in X \mid d(x,y)<r\}$ and $S_r^X(x)= \{ y\in X \mid d(x,y) = r\}$.  When no confusion will arise from suppressing the superscript in our notations for balls and spheres, we will do so.  If $X$ is a planar set then $B_r^X(x) = B_r^{\mathbb R^2}(x)\cap X$ and $S_r^X(x) = S_r^{\mathbb R^2}(x)\cap X$. For non-degenerate sets $A\subset X$, we will use $\mathcal N_\epsilon (A) = \{x\in X \mid d(x,a)<\epsilon \text{ for some } a\in A\}$ to denote the  $\epsilon$-neighborhood of  $A$ in $X$.

If $U\subset X$, we will use $\cl_X(U)$ to denote the topological closure of $U$  in $X$.  When $X$ is understood, the closure will be denoted simply by $\cl(U)$.  We will denote the unit circle in the plane by $\mathbb S^1$ and the unit sphere in $\mathbb R^3$ by $\mathbb S^2$.

Let $\alpha: \bigl([0,a], 0,a\bigr) \to \bigl(X, x_0, x_1\bigr)$ be a path.  Then $\overline \alpha: \bigl([0,a], 0,a\bigr) \to \bigl(X, x_1, x_0\bigr)$ is the path defined by $\overline \alpha (t) = \alpha(a-t)$.  A path $\alpha:\bigl([0,a], 0,a\bigr) \to \bigl(X, x_0, x_1\bigr)$ induces a change of base point isomorphism $\widehat \alpha : \pi_1(X, x_0) \to \pi_1(X,x_1)$ defined by $\widehat{\alpha}([s]) = [\overline\alpha* s* \alpha]$.

A space is \emph{locally simply connected at $x$} if every neighborhood of $x$ contains a simply connected neighborhood of $x$.  A space $X$ is \emph{semilocally simply connected at $x$} if there exists a neighborhood $U$ of $x$ such that the inclusion induced homomorphism $i_*:\pi_1(U,x)\to \pi_1(X,x)$ is trivial.  A space is \emph{locally simply connected} (or \emph{semilocally simply connected}) if it is at each its points.

\end{defn}

\begin{defn}
Let $A, B$ be subsets of a topological space $X$.  A closed subset $L$ of $X$ \emph{separates} $A$ and $B$, if there exists disjoint open subsets $U,V$ of $X$ such that $A\subset U$,  $B\subset V$, and $X\backslash L = U\cup V$.
\end{defn}

\begin{thm}\cite[Corollary 4.5.12]{vm} Let $X$ be a  non-empty topological space.  Then the covering dimension of $X$ is at most $1$ if and only if for every pair of disjoint closed subsets $A,B$ of $X$ there exists a closed 0-dimensional subset $L$ of $X$ that separates $A$ from $B$.
 \end{thm}

\begin{lem}[{\cite[Lemma 5.5]{ConnerKent19}}]\label{out}

Suppose that $X$ is a topological space and $x\in X$ has a planar or one-dimensional open neighborhood.  Then every open neighborhood $W$ of $x$ contains an open neighborhood $U$ of $x$ such that no essential loop in $U$ can be freely homotoped out of $W$.

\end{lem}

 This also proves the following well-known fact.

\begin{cor}
A planar or one-dimensional set is semilocally simply connected if and only if it is locally simply connected.
\end{cor}

\begin{lem}\label{complement}
Let $A$ be a closed connected subset of $\mathbb S^2$.  Then each component of $\mathbb S^2\backslash A$ is simply connected.
\end{lem}

\begin{proof}[Sketch of proof.]
Let $U$ be a component of $\mathbb S^2\backslash A$ and suppose that $\alpha:\mathbb S^1\to U$ is a loop in $U$.  Since $A$ and $\im(\alpha)$ are disjoint compact sets, $\alpha$ is homotopic in $U$ by a straight line homotopy to a polygonal path $\beta$.  Suppose that $\beta'$ is a simple closed subpath of $\beta$.  Since $A$ is connected exactly one of the components of $\mathbb S^2\backslash \im(\beta')$ can intersect $A$ and the component not intersecting $A$ must be contained in $U$ which implies that $\beta'$ is nullhomotopic in $U$.

Since $\beta$ is a polygonal path it can be reduced to the constant path by a finite process of replacing simply closed subpaths by constant paths.  By the previous argument, this process preserves the homotopy class and thus $\beta$ is nullhomotopic in $U$.
\end{proof}

When we say \emph{$A$ separates $B$ in $C$}, we mean that $B\subset C\backslash A$  and $B$ is not contained in a single connected component of $C\backslash A$.

\begin{lem}\label{separate} Let $X$ be a codiscrete subset of $\mathbb S^2$ and $A\subset X\backslash B(X)$  such that $ \cll(A)\cap D(X)= \emptyset$.  Then $\cll(A) = \cl_X(A)$ and if $A$ separates $B(X)$ in $X$ then $A$ separates $D(X)$ in $\mathbb S^2$.
\end{lem}

\begin{proof}
Since $\cll(A)\backslash \cl_X(A) \subset D(X)$ and  $\cll(A)\cap D(X) =\emptyset$, we have $\cll(A) = \cl_X(A)$.

Since $A$ separates $B(X)$ in $X$ there exists a continuous function $h: X \backslash A \to \{0,1\}$ which is non-constant on $B(X)$.  For each $d\in D(X)$ there exists $\epsilon_d>0$ such that  $\bigl\{B^{\mathbb S^2}_{\epsilon_d}(d)\mid d\in D(X) \bigr\}$ is a cover of $D(X)$ by disjoint open balls each of which intersects $D(X)$ at a unique point and is disjoint from $A$.  Thus $h$ is constant on $B^{\mathbb S^2}_{\epsilon_d}(d)\backslash\{d\}$.  Then we can define $\bar h : \mathbb S^2 \backslash A \to \{0,1\}$ by $\bar h(x) = h(x)$ on $X\backslash A$ and making $\bar h$ constant on each ball $B^{\mathbb S^2}_{\epsilon_d}(d)$.  It is immediate that $\bar h$ is continuous.  Since $h$ was non-constant on $B(X)$  so is $\bar h$.  Then $\bar h^{-1}(0)$ and $\bar h^{-1}(1)$ are disjoint open sets and both intersect $B(X)$ which implies they both intersect $D(X)$.  Hence $\bar h$ is non-constant on $D(X)$.
\end{proof}

\begin{lem}\label{simplyclosedcurve}

Let $X$ be a codiscrete subset of $\mathbb S^2$ and $U$ an open subset of $\mathbb S^2$ whose closure is a proper subset of $\mathbb S^2$.  Fix $d\in U$ and $c\in \mathbb S^2\backslash \cll(U)$.  For every $\epsilon$ there exists a simply closed curve in $U\cap \mathcal N_\epsilon (\partial U)\cap X$ which is essential in $\mathbb S^2\backslash \{c,d\}$.

\end{lem}

\begin{proof}
Fix $d\in U$ and $c\in \mathbb S^2\backslash \cll( U)$ and $\epsilon>0$.  Let $A$ be a component of $\partial U$ which separates $c,d$.    Let $V$ be the component of $\mathbb S^2\backslash A$ which contains $d$.  By Lemma \ref{separate}, $V$ is simply connected.   After possible passing to a subset of $A$ which still  separates $c$ and $d$, we may assume that $A =\partial V$.  By the Riemann Mapping Theorem there exists a homeomorphism $f: V \to \mathbb R^2$.  We may assume that $f(d) = (0,0)$ and $2\epsilon <d(A, d)$.

Then $C = V\backslash \mathcal N_\epsilon(A)$ is a compact set which contains $d$, hence $f(C)$ is also compact.  We can find an $M>1$ such that $f(C)\subset B_M\bigl((0,0)\bigr)= B$.

Then $f^{-1}(\partial B)$ is a simple closed curve which separates $d,c$ and $f^{-1}(\partial B)\subset \mathcal N_\epsilon (A) \subset \mathcal N_\epsilon (\partial U)$.   Notice that $f^{-1}(\partial B)$ might  intersect $D(X)$.  However after perturbing $f^{-1}(\partial B)$ near its intersections with $D(X)$,  we may also assume that it is contained in both $X$ and $\mathcal N_\epsilon (A)$.  (However, $f^{-1}(\partial B)$ can not necessarily be homotoped off of $B(X)$.)

\end{proof}

We will require the following technical lemma concerning separating components of point preimages.

\begin{lem}\label{kernel}

Suppose that $f: X \to Y$ is a continuous map where $X$ is a codiscrete subset of $\mathbb S^2$.  Fix $y\in Y$ such that $f^{-1}(y)$  separates $D(X)$ in $\mathbb S^2$.    Then either $\cll\bigl(A\bigr)\cap D(X)\neq \emptyset$ for every component $A$ of $f^{-1}(y)$ which separates $D(X)$ in $X$ and is disjoint from $B(X)$ or $f_*$ the induced map on fundamental groups is not injective.

\end{lem}

\begin{proof}
	Fix $y\in Y$ such that $f^{-1}(y)$ separates $D(X)$ in $\mathbb S^2$.  Suppose that $A$ is a component of $f^{-1}(y)$ which separates $D(X)$ in $\mathbb S^2$, is disjoint from $B(X)$, and $\cll\bigl(A\bigr)\cap D(X)= \emptyset$.

		Notice that $\cll\bigl(A\bigr)\cap D(X)= \emptyset$ implies that $ \cll\bigl(A\bigr) = \cl_X(A)$. Since $A$ is a component of $f^{-1}(y)$, it is a maximal connected subset of the closed set $f^{-1}(y)$ which implies that $\cl_X(A) \subset A$.  Hence $A$ is a connected closed subset of $\mathbb S^2$.

	Since $A \cap B(X) = A\cap D(X)= \emptyset$, there exists $\epsilon>0$ such that $\mathcal N_\epsilon (A) \cap D(X) = \emptyset$.  Let  $U$ be a component of $\mathbb S^2\backslash A$ such that $U\cap D(X) \neq \emptyset$ and $\bigl(\mathbb S^2\backslash \cll(U)\bigr)\cap D(X)\neq \emptyset$.  Let $c\in U\cap D(X) $ and $d\in \bigl(\mathbb S^2\backslash \cll(U\bigr)\cap D(X)$. This is possible since $A$ separates $D(X)$ in $\mathbb S^2$.

 By Lemma \ref{simplyclosedcurve}, there exists a simple closed curve $\alpha: \mathbb S^1 \to U\cap\mathcal N_{\epsilon/2} (A)\cap X $ which is homotopically essential in $\mathbb S^2\backslash \{c,d\}$.  Hence $\alpha$ is also homotopically essential in $X$.
	
	Let $D$ be the disc in $\mathbb S^2$ with boundary parameterized by $\alpha$ that does not contain $c$.   By construction, the components of $D\backslash A$ are either contained in $X$ or have boundary contained in $A$. Hence the map $f\circ\alpha:\mathbb S^1 \to Y$ extends to a map $h:D\to Y$ by sending each of the components bounded by $A$ to $y$ and letting $h$ agree with $f$ on the components contained in $X$.  Thus $f\circ\alpha$ extends to a map of the disc and $f_*$ has a nontrivial kernel.

\end{proof}

We will use the following theorem that arbitrary homomorphisms of the fundamental groups of planar and one-dimensional Peano continua are induced by continuous maps.

\begin{thm}[\cite{eda}, \cite{ConnerKent19},\cite{Kent18}]\label{continuous}

Let $X$ and $Y$ be one-dimensional or planar Peano continua and $\phi:\pi_1(X,x_0) \to\pi_1(Y,y_0)$ a homomorphism of their fundamental groups. Then there exists a continuous function $f:X \to Y$ and a path $T:(I,0,1)\to (Y,y_0,y)$ such that $f_* =\widehat T \circ\phi$ where $\widehat T$ is the change of base-point isomorphism induced by $T$.

\end{thm}

When $X$ and $Y$ are one-dimensional this was proved by Eda.  When $X$ is one-dimensional and $Y$ is planar this was done by Conner and Kent.  When $X$ is planar and $Y$ is one-dimensional or planar this was proved by Kent.

The following lemma is an easy exercise which will be left to the reader.

\begin{lem}\label{factordendrite}

Suppose that $\alpha: \mathbb S^1 \to A\subset \mathbb S^2$ is a continuous function and $f: A\to \mathbb S^2\backslash\{c,d\}$ is a continuous map such that $d\bigl(\alpha(t), f\circ\alpha(t)\bigr)< d\bigl(\alpha(t), \{c,d\}\bigr)$ where the metric on the 2-sphere is the standard CAT(1) metric.    If $f\circ\alpha$ is nullhomotopic, then $\alpha$ is nullhomotopic in $\mathbb S^2\backslash\{c,d\}$.

\end{lem}

\begin{lem}[{\cite[Lemma 2.10]{ConnerKent19}}]\label{conjugate} Let $X$ be a one-dimensional or planar topological space.  Suppose that  $\alpha_n$ and $\beta_n$ are two  null sequences of essential loops based at $x_0$ and $x_1$ respectively.  If $\alpha_n$ is freely homotopic to $\beta_n$ for all $n$, then $x_0 = x_1$.  If $X$ is also a Peano continuum and there exists a loop $\gamma$ such that $\overline\gamma*\alpha_n*\gamma$ is homotopic rel endpoints to $\beta_n$ for all $n$, then $\gamma$ is a nullhomotopic loop.\end{lem}

\begin{prop}\label{inner}
Let $X$ be a one-dimensional or planar Peano continuum and fix $x_0\in B(X)$.  If $\phi: \pi_1(X,x_0) \to \pi_1(X,x_0)$ is a nontrivial inner automorphism, then $\phi$ is not induced by a continuous function.
\end{prop}


\begin{proof}
Suppose that there exists a continuous function $f: X\to X$ such that $f_*: \pi_1(X,x_0) \to \pi_1(X,x_0)$ is an inner automorphism.  Then there exists a loop $\gamma$ at $x_0$ such that $f_*([s]) = \hat\gamma\bigl([s]\bigr) = [\overline\gamma * s* \gamma]$ for all loops $s$ based at $x_0$.  Since $x_0\in B(X)$, there exists a null sequence of essential loops $\alpha_n$ at $x_0$.  Then $\alpha_n$ and $f\circ \alpha_n$ are null sequences of loops which are conjugate by $\gamma$;  hence Lemma \ref{conjugate} implies that $\gamma$ is nullhomotopic.  Thus $f_*$ is the identity homomorphism.

\end{proof}

\begin{lem}\label{homotopy equivalence}Let $\bar X$ be a one-dimensional or planar Peano continuum and $X$ a planar or one-dimensional topological space.  Suppose that $k:X\to \bar X$ is a homotopy equivalence with homotopy inverse $h:\bar X\to X$.  Then $k\circ h|_{B(\bar X)} = id_{B(\bar X)}$ and, for every $x\in B(\bar X)$, $(k\circ h)_*= id_{\pi_1\bigl(\bar X, k(x)\bigr)}$.  Similarly $h\circ k|_{B( X)} = id_{B( X)}$ and, for every $x\in B(X)$, $(h\circ k)_*=id_{\pi_1\bigl( X, x\bigr)}$.

\end{lem}

\begin{proof}Let $Y$ be any planar or one-dimensional space and $y_0\in B(Y)$.  Suppose that $H: Y\times I \to Y$ is a continuous map such that $H|_{Y\times\{0\}}$ is the identity on $Y$.  Then $H(y_0,t) = y_0$ for all $t$ by Lemma \ref{conjugate}.  Thus  $k\circ h|_{B(\bar X)} = id_{B(\bar X)}$ and $h\circ k|_{B( X)} = id_{B( X)}$.

Then \cite[Lemma 1.19]{Hatcher} together with Lemma \ref{conjugate} imply that $(k\circ h)_*$ and $(h\circ k)_*$ are  the identity automorphisms on $\pi_1\bigl(\bar X, k(x)\bigr)$ and  $\pi_1\bigl(X, x\bigr)$, respectively. (In the notation of \cite{Hatcher}, the change of basepoint isomorphism $\beta_h$ of Lemma 1.19 must be the constant path by Lemma \ref{conjugate}.)
\end{proof}

\begin{lem}\label{inner2}Let $\bar X$ be a one-dimensional or planar Peano continuum and $X$ a planar or one-dimensional topological space.  Suppose that $k:X\to \bar X$ and $h:\bar X\to X$ are homotopy equivalences.  If $x\in B(X)$ and $\phi: \pi_1(X, x) \to \pi_1(X,x)$ is an inner automorphism, then $k_*\circ \phi\circ h_*$ is an inner automorphism of $\pi_1\bigr(\bar X, k(x)\bigl).$  If, in addition, $k_*\circ\phi\circ h_*$ is the identity on $\pi_1\bigl(\bar X, k(x)\bigr)$, then $\phi$ is the identity homomorphism on $\pi_1(X, x)$.
\end{lem}

\begin{proof}If $\phi$ is an inner automorphism, then there exists $[\gamma]\in \pi_1(X,x)$ such that $\phi([s]) = \hat\gamma\bigl([s]\bigr) = [\overline\gamma * s* \gamma]$ for all loops $s$ based at $x$.  For any loop $\alpha$ based at $k(x)$, we have that $k\circ h\circ \alpha$ is homotopic rel endpoints to $\alpha$ and $ k_*\circ \phi\circ h_*\bigl([\alpha]\bigr) =k_*\circ \phi \bigl([h\circ \alpha]\bigr) = k_*\bigl( [\overline\gamma * (h\circ \alpha)* \gamma]= \bigl[(k\circ \overline\gamma) * (k\circ h\circ \alpha)* (k \circ\gamma)\bigr] = \bigl[(\overline{k\circ \gamma} )*  \alpha*( k \circ\gamma)\bigr]$.  Thus $ k_*\circ \phi\circ h_*$ is an inner automorphism.

Suppose that, in addition to $\phi$ being an inner automorphism,  $k_*\circ\phi\circ h_*$ is the identity on $\pi_1\bigl(\bar X, k(x)\bigr)$.  Then for any loop $\alpha$ based at $k(x)$, we have that $ [\alpha]= k_*\circ \phi\circ h_*\bigl([\alpha]\bigr) = \bigl[(\overline{k\circ \gamma} )*  \alpha* (k \circ\gamma)\bigr]$.  Lemma \ref{conjugate} implies that $k\circ \gamma$ is nullhomotopic.  Since $k$ is a homotopy equivalence, $k_*$ is injective and $\gamma$ must by nullhomotopic.  Thus $\phi$ is the identity homomorphism.
 \end{proof}

\begin{lem}\label{induced inverses}
If $X,Y$ are planar or one-dimensional Peano continua with isomorphic fundamental groups then there exit continuous maps $f: X \to Y$ and $g: Y\to X$ such that $g\circ f(x_0) = x_0$  for some $x_0\in X$ and $(g\circ f)_*= id_{\pi_1(X,x_0)}$.
\end{lem}

\begin{proof}A planar or one-dimensional Peano continua is locally simply connected if and only if it has countable fundamental group, see \cite[Theorem 3.1]{ConnerLamoreaux2005} in the planar case and \cite[Theorem 5.9]{cc3} in the one-dimensional case.  Thus if $B(X)= \emptyset$, then $B(Y) = \emptyset$ and both $X$ and $Y$ are homotopy equivalent to a finite wedge of circles in which case the lemma is standard.

Thus we need only concern ourselves with the case that $B(X) \neq\emptyset$.  Fix $x_0\in B(X)$.   By Theorem \ref{continuous} there exists a continuous function $f:( X,x_0)\to (Y, y_0)$ such that $f_{*}$ is an isomorphism.   We can apply Theorem \ref{continuous} again to find continuous maps $\alpha: (I,0,1) \to ( X, x_0, x_1)$ and $g: Y\to  X$ such that
$f_{*}^{-1} =  \widehat{\overline\alpha}\circ g_{*}$.  Then $\widehat{\alpha} = (g\circ f)_*: \pi_1(X, x_0) \to \pi_1(X, x_1)$ is a change of basepoint isomorphism.  Hence $(g\circ f)_*$ defines an isomorphism from $\pi_1(X,x)$ to $\pi_1(X, g\circ f(x))$ for all $x\in X$.

Let $s_n: I \to X$ be a null sequence of essential loops based at $x_0$.   Notice that $\overline \alpha* s_n* \alpha$ is homotopic to $g\circ f\circ s_n$ and freely homotopic to $s_n$. Applying Lemma \ref{conjugate} to $g\circ f\circ s_n$ and $s_n$  gives that $x_1= x_0$.  Thus $\alpha$ is a loop based at $x_0$ and $(g\circ f)_*$ is an inner automorphism.  Lemma \ref{inner} implies that $(g\circ f)_*$ is the identity on $\pi_1(X,x_0)$ as desired.

\end{proof}

The following is a technical lemma from Cannon and Conner's proof of Theorem \ref{chara} that is a slight strengthening of condition (\ref{ii}), see \cite[Lemma 4.3]{cc4}.

 \begin{lem}[Annulus Lemma]\label{annulus}

 Suppose that $X$ is a codiscrete subset of the 2-sphere for which condition (\ref{ii}) of Theorem \ref{chara} fails.  Then there exists a closed annulus $A$ in $\mathbb S^2$ and infinitely many components $U_1,U_2, \cdots$ of $A \backslash B(X)$ such that $U_i$ intersect both boundary components of $A$ and $U_i \cap D(X) = \emptyset$ for all $i\in \mathbb N$.

 \end{lem}

\section{Proof of Theorem \ref{almostmain}}\label{proof}

\begin{thm}\label{almostmain}

Let $X$ be a codiscrete subset of the 2-sphere, $\mathbb S^2$. Then $X$ is homotopy equivalent to a one-dimensional Peano continuum if and only if $\pi_1(X,x_0)$ is isomorphic to the fundamental group of a one-dimensional planar Peano continuum.

\end{thm}

This section will be dedicated to the proof of Theorem \ref{almostmain}.  If $X$ is homotopy equivalent to a one-dimensional Peano continuum, then any homotopy equivalence induces an isomorphism of fundamental groups.  Thus we need only prove that if the fundamental group of a codiscrete set $X$ is isomorphic to the fundamental group of a one-dimensional Peano continuum then $X$ is homotopy equivalent to a one-dimensional Peano continuum.

For the remainder of the section, we will fix  a codiscrete subset $X$ of $\mathbb S^2$ and a one-dimensional Peano continuum $Y$ such that $\pi_1(X,x)$ is isomorphic to $\pi_1(Y,y)$.  Note that if $X$ is locally simply connected, then $X$ is a finitely punctured sphere and is homotopy equivalent to a bouquet of finitely many circles.  Thus the theorem is trivial if $X$ is locally simply connected.  Thus we may assume that $X$ is not locally simply connected and fix $x_0\in B(X)$.

\begin{lem}\label{inverse}

There exists continuous maps $f: X \to Y$ and $g: Y\to X$ such that $g\circ f(x_0) = x_0$  and $(g\circ f)_*= id_{\pi_1(X,x_0)}$.

\end{lem}

\begin{proof}

By Theorem \ref{classification} exists $\bar X$ a planar Peano continuum and homotopy equivalences $h: \bar X \to X$ and $k:X \to \bar X$. Since $x_0\in B(X)$, $h\circ k(x_0) = x_0$ by Lemma \ref{homotopy equivalence}.

By Lemma \ref{induced inverses}, there exists maps $f_1: \bar X \to Y$ and $g_1: Y \to \bar X$ such that $g_1\circ f_1 \bigl(k(x_0)\bigr) = k(x_0)$ and $(g_1\circ f_1)_* = id_{\pi_1\bigl(\bar X, k(x_0)\bigr)}$.

Then $h\circ g_1\circ f_1\circ k(x_0) = x_0$ and $(h\circ g_1\circ f_1\circ k)_* = h_*\circ g_{1*}\circ f_{1*}\circ k_* = h_*\circ k_* $.  By Lemma \ref{homotopy equivalence} $(h\circ k)_*= id_{\pi_1(X,x_0)}$.  So $f_1\circ k$ and $h\circ g_1$ are the desired maps.
\end{proof}

Using the notation from Lemma \ref{inverse}, we will now fix $f_1\circ k = f: X \to Y$ and $h\circ g_1 =g: Y\to X$ such that $g\circ f(x_0) = x_0$  and $(g\circ f)_*= id_{\pi_1(X,x_0)}$ where $h: \bar X \to X$ and $k:X \to \bar X$ are homotopy equivalences of $X$ with a planar Peano continuum.

\begin{lem}\label{identity}

 If $x\in B(X)$, then  $g\circ f (x) = x$, any path $\alpha: (I,0,1) \to (X,x_0, x)$ is homotopic rel endpoints to $g\circ f\circ \alpha $, and $g\circ f$ induces the identity homomorphism on $\pi_1(X,x)$.

 \end{lem}

\begin{proof}
Fix $x\in B(X)$ and a path  $\alpha: (I,0,1) \to (X,x_0, x)$. By Lemma \ref{conjugate}, we have that  $g\circ f(x) = x$. Let $s: I \to X$ be an essential loop based at $x$.  Since $(g\circ f)_*$ is the identity on $\pi_1(X, x_0)$ we have $(g\circ f\circ\alpha)* (g\circ f \circ s) * (g\circ f\circ\overline \alpha)$ is homotopic  rel endpoints to  $\alpha*s*\overline \alpha$.  Hence $g\circ f\circ s$ is homotopic rel endpoints to $(g\circ f\circ\overline\alpha)*\alpha*s*\overline \alpha*(g\circ f\circ\alpha)$. Thus $(g\circ f)_*$ defines an inner automorphism of $\pi_1(X, x)$, which by Lemma \ref{inner2} implies that $ k_*\circ (g\circ f)_*\circ h_*$ is an inner automorphism.  Since $ k_*\circ (g\circ f)_*\circ h_*$ is induced by a continuous function, it is the identity on $\pi_1\bigl(\bar X, k(x)\bigr)$ by Lemma \ref{inner2}.  By  Lemma \ref{inner}, $g\circ f$ induces the identity homomorphism on $\pi_1(X,x)$.  Hence $g\circ f\circ\overline\alpha*\alpha$ is nullhomotopic and $g\circ f\circ\alpha$ is homotopic rel endpoints to $\alpha$.

\end{proof}

\begin{lem}\label{part 1}

The codiscrete space $X$ satisfies condition (\ref{i}) of Theorem \ref{chara}.

\end{lem}

\begin{proof}
Suppose there exists a component $U$ of $\mathbb S^2 \backslash B(X)$ such that $U \cap D(X) =\emptyset$.  Notice that this implies that $\cl_X(U)$ is compact.

\textbf{Case 1: $U$ is not simply connected.} Then $U$ has at least two distinct boundary components.  Let $B(X) = D_1\cup D_2$ where $D_1,D_2$ are disjoint nonempty closed sets and $\partial U$ intersects both sets.

Then $f(D_1),f(D_2)$ are disjoint closed subsets of a one-dimensional Peano continuum.  Hence there exists a $0$-dimensional subspace $L$ of $Y$ which separates $f(D_1),f(D_2)$.  Let $C_i$ be a boundary component of $U$ contained in $D_i$ for $i= 1,2$.

Since $\cl(U)$ is connected any set which separates boundary components of $U$ must intersect $U$.  Since $f^{-1}(L)$ is disjoint from $\partial U$ any component of $f^{-1}(L)$ which intersects $U$ must be contained in $U$.  Therefore there exists a component $A$ of $f^{-1}(L)$ which separates $C_1, C_2$ and  is contained in $U$.

By hypothesis $U\cap D(X) = \emptyset$ and $\partial U\subset B(X)$ which implies that $\cll(A) \cap D(x)=\emptyset$.  Then Lemma \ref{separate} implies that $A$ separates $D(X)$ in $\mathbb S^2$.

Since $L$ is $0$-dimensional and $A$ is connected, $f$ is constant on $A$ and Lemma \ref{kernel} implies that $f_*$ is not injective which is a contradiction.

\textbf{Case 2: $U$ is simply connected.}  Fix $x_1\in U$ and $\epsilon= d\bigl(x_1, B(X)\bigr)>0$.  Since $U\cap D(X) = \emptyset$ and $\partial U\subset B(X)$, we have that $B(X)$ separates $x_1$ and any point of $D(X)$.  Thus $d\bigl(x_1, D(X)\bigr) \geq d\bigl(x_1, B(X)\bigr)$. Choose $0<\delta<\epsilon/4$ such that $d\bigl(g\circ f(x), g\circ f(y)\bigr)<\epsilon/4$ for every $x,y\in \cl_X(U)$ with $d(x,y)<\delta$.

By Lemma \ref{simplyclosedcurve}, there exists a simply closed curve $\alpha: \mathbb S^1 \to \mathcal N_\delta(\partial U)\cap U$ which is homotopically essential in $\mathbb S^2\backslash\{c,x_1\}$ for any fixed $c\in D(X)$.

For every $x\in \im(\alpha)$ there exists a $u \in B(X)$ such that $d(x,u)< \delta$ and we have

\begin{align*}
	d\bigl(x, g\circ f(x)\bigr) &\leq d(x,u) + d\bigl(u, g\circ f(u)\bigr) + d\bigl( g\circ f(u), g\circ f(x)\bigr) \\ &
	< \epsilon/4 + 0 + \epsilon/ 4 \leq \epsilon/2< d\bigr(x, \{x_1, D(X)\}\bigl).
\end{align*}

Since $\alpha$ bounds a disc in  $X$, $f\circ\alpha$ must factor through a dendrite and hence so does $g\circ f\circ \alpha$.   Thus $g\circ f\circ\alpha$ is nullhomotopic in $\im(g\circ f \circ \alpha) \subset \mathbb S^2\backslash \{x_1,c\}$. However $\alpha$ is homotopically essential in $\mathbb S^2\backslash \{x_1,c\}$ which contradicts Lemma \ref{factordendrite}.

\end{proof}

\begin{lem}\label{small}
Suppose that $J$ is a simply closed curve in $\mathbb S^2$ and $W_n$ is a sequence of open subsets of  $J$ such that $W_n$ is contained in a connected component of $J\backslash \bigcup\limits_{i\neq n} \cl( W_i)$.  Then $\diam(W_n)$ converges to $0$.
\end{lem}

\begin{proof}
  Let $h: \mathbb S^1 \to J\subset \mathbb S^2$ be a homeomorphism.  Fix  a sequence of open subsets $W_n$ of  $J$ such that $W_i$ is contained in a connected component of $J\backslash \bigcup\limits_{i\neq n} \cl( W_i)$.  Let $U_n $ be the connected component of $J\backslash \bigcup\limits_{i\neq n} \cl\bigl( W_i\bigr)$ containing $W_n$.

  Since $J$  is compact, it can only contain finitely many disjoint open intervals of diameter greater than any fixed $\delta$. Thus $\diam (U_n)\leq \delta$  for all but finitely many $n$, which implies that $\diam(W_n)\leq \delta$ for all but finitely many $n$.
\end{proof}

\begin{lem}\label{part 2}

$X$ satisfies condition (\ref{ii}) of Theorem \ref{chara}.

\end{lem}

\begin{proof}
By way of contradiction, suppose that $X$ does not satisfy (\ref{ii}) of Theorem \ref{chara}.

By Lemma \ref{annulus}, there exists a closed annulus $A$ in $\mathbb S^2$ and components $U_1,U_2, \cdots$ of $A \backslash B(X)$ such that $U_n$ interests both boundary components of $A$ and $U_n \cap D(X) = \emptyset$ for all $n\in \mathbb N$. Let $J_1$ and $J_2$ be the boundary components of $A$.

Since $U_n\cap D(X)= \emptyset$ and $U_n$ is a component of $A\backslash B(X)$, we have that $\cl_{\mathbb S^2}(U_n)\cap D(X) = \emptyset$.  This implies that $\cll(U_n) = \cl_X(U_n)$ and hence $\cl_X\bigl(\bigcup\limits_n U_n\bigr)$ is compact.  (This follows since $\cl_X\bigl(\bigcup\limits_n U_n\bigr)\backslash \bigl(\bigcup\limits_n U_n\bigr) \subset B(X)$.)

The boundary of $U_n$ has two types of points.

\begin{enumerate}[{Type} (1)]
  \item Points that are contained in $B(X)$.  When restricted to these points, the map $f$ is a homeomorphism and $g\circ f$ is the identity.\vspace{1em}
  \item Points which are on $J_1\cup J_2\backslash B(X)$.  A priori we have not control over what $f$ or $g\circ f$ does on these types of points.  So we have to insure that for sufficiently large $n$ the diameter of the set of points in $U_n$ of this type is sufficiently small.
\end{enumerate}


Since each $U_n$ is open and has closure intersecting both $J_1$ and $J_2$, there exists an $\epsilon_1>0$ and a sequence of points $x_n \in U_n$ such that $d\bigl(x_n, J_1\cup J_2\bigr) \geq 4\epsilon_1$ for some $\epsilon_1>0$.  We may also fix $c\in D(X)$ such that $d\bigl(c, U_n\bigr)>4\epsilon_1$ for all $n$ (by possible passing to a cofinal subsequence of $U_n$ and choosing a smaller $\epsilon_1$).

Fix $0<\delta_1<\epsilon_1/4 $ such that $d\bigl(g\circ f (x), g\circ f(y)\bigr) < \epsilon_1$ for all $x,y \in \cl_X\Bigl(\bigcup\limits_n U_n\Bigr)$ with $d(x,y)<4\delta_1$.

Fix $0<\delta_2< \delta_1$ such that  $d\bigl(g\circ f (x), g\circ f(y)\bigr) < \delta_1$ for all $x,y \in \cl_X\Bigl(\bigcup\limits_n U_n\Bigr)$ with $d(x,y)<\delta_2$.

Since each $U_i$ is open and connected in $A$ and intersects both boundary components of A,  $U_i\cap J_k$ is contained in a connected component of $J_k\backslash \bigcup\limits_{i\neq n} \cl( U_i)$ for $k = 1,2$.  The by Lemma \ref{small} $\diam(U_n\cap J_1)$ and $ \diam (U_n\cap J_2)$ both converge to $0$, i.e. the set of points of $U_n$ of Type (2) can be partitioned into two subsets each of which has small diameter.  Hence we can fix an $n_0$ such that $\max\bigl\{\diam(U_{n_0}\cap J_1),\diam(U_{n_0}\cap J_2)\bigr\}< \delta_2$.  Let  $W$ be the component of $U_{n_0}\backslash \cl_X\bigl(\mathcal N_{2\delta_1} (\partial A)\bigr)$ containing $x_{n_0}$.

\textbf{Case 1: $W$ is not simply connected.}  Then $U_{n_0}$ has a boundary component $C_1$ such $d\bigl(C_1, x \bigr)\geq 2\delta_1$ for all points $x$ of Type 2.  Let $C_2$ be the component of $\partial U_{n_0}$ which intersects both $J_1$ and $J_2$.  Since $\max\bigl\{\diam(U_{n_0}\cap J_1),\diam(U_{n_0}\cap J_2)\bigr\}< \delta_2$ and $f$ restricts to a homeomorphism on $B(X)$, we see that $f(C_1)$ and $f(C_2)$ are disjoint closed sets.

 Then $f\bigl(\partial U_{n_0}\bigr)$ has at least two distinct components.  Let $\partial U_{n_0} = D_1\cup D_2$ where $D_1,D_2$ are disjoint nonempty closed sets with $C_i\subset D_i$ for $i =1,2$ and $f(D_1)\cap f(D_2)=\emptyset$.

Hence there exists a $0$-dimensional subspace $L$ of $Y$ which separates $f(D_1),f(D_2)$. Then $f^{-1}(L)$ separates $C_1, C_2$  which implies that a single component $E$ of $f^{-1}(L)$ separates $C_1, C_2$.  Since $f^{-1}(L)$ is disjoint from $\partial U_{n_0}$, $E$ is a subset $U_{n_0}$ and $E\cap B(X)= \emptyset$.  By hypothesis $U_{n_0}\cap D(X) = \emptyset$ and $\partial U_{n_0}\subset X$ which implies that $\cll(E) \cap D(X)=\emptyset$.

By construction, $E$ separates $B(X)$ so Lemma \ref{separate}  implies that $E$ separates $D(X)$ in $\mathbb S^2$.  Since $L$ is $0$-dimensional and $E$ is connected, $f$ is constant on $E$ and Lemma \ref{kernel} implies that $f_*$ is not injective which is a contradiction.

\textbf{Case 2: $W$ is simply connected.}

Let   $4\epsilon_2 = \min\bigl\{d(x_{n_0}, \partial U_{n_0}), 4\epsilon_1\bigr\}$.  We can choose $0<\delta_3\leq \min\{\epsilon_2, \delta_2\}$ such that  $4d\bigl(g\circ f (x), g\circ f(y)\bigr) < \epsilon_2$ for all $x,y\in \cll(U_{n_0})$ with $d(x,y)< \delta_3$.

By Lemma \ref{simplyclosedcurve}, there exists a simply closed curve $\alpha: \mathbb S^1 \to \mathcal N_{\delta_3}(\partial W)\cap W$ which is homotopically essential in $\mathbb S^2\backslash\{c,x_{n_0}\}$.  Since $\im(\alpha)\subset U_{n_0}$ and $d(U_{n_0}, c)\geq 4\epsilon_1$  we have that $d\bigl(\im(\alpha), c\bigr) \geq 4\epsilon_1$.  Points on $\partial W$ are on $\partial U_{n_0}$ or at least $4\epsilon_1- 2\delta_1$ from $x_{n_0}$. We now need to show that $d\bigl(x, g\circ f (x)\bigr) < d\bigl(x, \{x_{n_0}, c\}\bigr)$ for every $x\in \im(\alpha)$. This breaks down into two case corresponding to whether $x$ is close to $\partial U_{n_0}$ or far from $x_{n_0}$ (at least $4\epsilon_1- 2\delta_1-\delta_3$).

\textbf{Subcase 1:} $d\bigl(x, \partial U_{n_0}\bigr) \leq \delta_3$.  Then $d\bigl(x, x_{n_0}\bigr) > 4\epsilon_2 - \delta_3$ and $u \in B(X)$ such that $d(x,u)\leq \delta_3$.  Thus

\begin{align*}
	d\bigl(x, g\circ f(x)\bigr) &\leq d(x,u) + d\bigl(u, g\circ f(u)\bigr) + d\bigl( g\circ f(u), g\circ f(x)\bigr) \\ &
	< \epsilon_2/4 + 0 + \epsilon_2/ 4 \leq \epsilon_2/2<  d\bigr(x, \{x_{n_0}, c\}\bigl).
\end{align*}

\textbf{Subcase 2:} $d\bigl(\alpha(t), \partial U_{n_0}\bigr) > \delta_3$. Then there exists $v\in \partial W$ such that $d(v,x)\leq \delta_3$ and $d(v, \partial A \cap \partial U_{n_0} ) \leq 2\delta_1$.  This implies that $d(x, x_{n_0}) \geq 4\epsilon_1- 2\delta_1-\delta_3 > 2\epsilon_1$.   Since $J_i \cap U_{n_0}$ has diameter at most $\delta_2$ for $i\in\{1,2\}$, there exists $u\in B(X)\cap \partial U_{n_0}\cap \partial A$ such that $d(x,u) \leq \delta_3 + \delta _2  + 2\delta_1< 4\delta_1$.  Thus

\begin{align*}
	d\bigl(x, g\circ f(x)\bigr) &\leq d(x,u) + d\bigl(u, g\circ f(u)\bigr) + d\bigl( g\circ f(u), g\circ f(x)\bigr) \\ &
	< \epsilon_1 + 0 + \epsilon_1 \leq 2 \epsilon_1< d\bigr(x, \{x_{n_0}, c\}\bigl).
\end{align*}

Since $\alpha$ bounds a disc in  $X$, $f\circ\alpha$ must factor through a dendrite and hence so does $g\circ f\circ \alpha$.  However $\alpha$ is homotopically essential in $\mathbb S^2\backslash \{x_{n_0},c\}$ which contradicts Corollary \ref{factordendrite}.

\end{proof}

Thus by Theorem \ref{chara} and Corollary \ref{cor: one-dimensional planar}, $X$ is homotopy equivalent to a one-dimensional planar Peano continuum, which completes the proof of Theorem \ref{almostmain}.


\section{Proof of Theorem \ref{main}}

\begin{lem}
    Let $X$ be a one-dimensional path-connected Hausdorff space and $A$ a path-connected subset of $X$ such that the inclusion map $i: A\to X$ induces an isomorphism.  Then any path $p:[0,1]\to X$ with $p\bigl((0,1)\bigr) \subset X\backslash A$ and $p(0),p(1)\in A$ is a nullhomotopic loop.
\end{lem}

\begin{proof}
  Fix a path $p:[0,1]\to X$ with $p\bigl((0,1\bigr) \subset X\backslash A$.  Let  $\tilde p$ be the reduced representative of $p$.  If $\tilde p$ is a non-degenerate loop, we may assume that $\tilde p \bigl((0,1)\bigr) \subset X\backslash A$.   Let $q:[0,1]\to A$ be any reduced path in $A$ from $\tilde p(1)$ to $\tilde p(0)$. Then  $\tilde p*q$ is a reduced path, since $\tilde p\bigl((0,1)\bigr) \cap q\bigl([0,1]\bigr)= \emptyset$ or $\tilde p$ is a degenerate loop.  Since the inclusion map $i: A\to X$ induces an isomorphism on fundamental groups, every essential reduced loop is contained in $A$.  Thus $\tilde p$ must be a degenerate loop and $p$ is a nullhomotopic loop.
\end{proof}

\begin{cor}
Let $X$ be a one-dimensional path-connected Hausdorff space and $A$ a path-connected closed subset of $X$ such that the inclusion map $i: A\to X$ induces an isomorphism.  Then every reduced path starting and ending in $A$ is contained in $A$.
\end{cor}

\begin{prop}\label{homotopic to image}
Let $X$ be a Peano continuum and $Y$ a one-dimensional Hausdorff space.  Suppose that $f: X\to Y$ is a homotopy equivalence.  Then $Y$ is homotopy equivalent to the one-dimensional Peano continuum $f(X)$.
\end{prop}

\begin{proof}

Let $f: X\to Y$ be a homotopy equivalence with homotopy inverse $g: Y \to X$.   Then there exists maps $F: X\times I \to X$ and $H: Y\times I \to Y$ such that $H(0,y) = y, F(0,x) = x$ and $H(1,y) = f\circ g(y), F(1,x) = g\circ f(x)$. Since $f$ is a homotopy equivalence, $f_*$ gives an isomorphism for any choice of base points.  Define $h: X \to f(X)$ by $h(x) = f(x)$ for all $x\in X$.  Then $i\circ h = f$ where $i: f(X)\to Y$ is the inclusion map.  Thus $i_*: \pi_1\bigl(f(X), f(x_0)\bigr)\to \pi_1\bigl(Y, f(x_0)\bigr)$ is surjective, since $f_*$ is surjective.  By \cite[Corollary 3.3]{cc3}, the inclusion induced homomorphism from $\pi_1\bigl(f(X), f(x_0)\bigr)$ to $\pi_1\bigl(Y,f(x_0)\bigr)$ is injective.  Thus $i_*$ is an isomorphism and $\pi_1\bigl(Y, f(x_0)\bigr)$ is isomorphic to $\pi_1\bigr(f(X), f(x_0)\bigl)$.  Thus every reduced path starting and ending in $f(X)$ is contained in $f(X)$.  Then for every $y\in f(X)$, the path $H(t,y)$ has reduced representative contained in $f(X)$.  By the Hahn-Mazurkiewicz Theorem, the continuous image of a Peano continuum in Hausdorff space is a Peano continuum.  Thus $f(X)$ is a Peano continuum and by \cite[Theorem 3.9]{ccz} there exists a parametrization $\alpha_y$ of $H(t,y)$ such that $\tilde H : f(X) \times I \to f(X)$ given by $H(y,t) = \alpha_y(t)$ is continuous.  Thus $f, g|_{f(X)}$ are homotopy inverses and $X$ is homotopy equivalent to $f(X)$.

\end{proof}

For convenience, we will restate Theorem \ref{main} and Theorem \ref{main2} before proving them.

\begin{customthm}{\ref{main}}\hspace{3in}
     \begin{enumerate}
        \item A planar Peano continuum with homotopy dimension one is homotopy equivalent to a one-dimensional planar Peano continuum. 

     	\item In the class comprised of the union of one-dimensional Peano continua and planar Peano continua, the fundamental group determines the homotopy dimension.
     \end{enumerate}
\end{customthm}

\begin{proof}[Proof of Theorem \ref{main}]\hspace{3in}

\emph{Statement \textup(1\textup):} Let $X$ be a planar Peano continuum and suppose that $X$ has homotopy dimension one.  Then $X$ is homotopy equivalent to a codiscrete subset $Y$ of the two-sphere which satisfies both conditions of Theorem \ref{classification}.  By Corollary \ref{cor: one-dimensional planar}, we have that $Y$ is homotopy equivalent to a one-dimensional planar continuum.  Thus Proposition \ref{homotopic to image} implies that $X$ is homotopy equivalent to a planar one-dimensional Peano continuum contained in $Y$.

\emph{Statement \textup(2\textup):}  A simply connected one-dimensional Peano continuum is a dendrite, which implies contractible.  Thus we need only consider planar Peano continua.  Suppose that $X$ is a planar Peano continuum and that there exists a one-dimensional Peano continuum $Y$  such that $\pi_1(X,x_0)$ is isomorphic to $\pi_1(Y,y_0)$. By Theorem \ref{classification}, $X$ is homotopy equivalent to a codiscrete subset of $\mathbb S^2$, which by Theorem \ref{almostmain} is homotopy equivalent to some one-dimensional Peano continuum.  Thus $X$ is homotopy equivalent to a one-dimensional Peano continuum.  If the one-dimensional Peano continuum is simply connected, then it is contractible, which would imply that $X$ is also contractible.

\end{proof}

\begin{cor}Let $X$ be a planar Peano continuum.  Then $X$ has homotopy dimension one if and only if the fundamental group of $X$ is isomorphic to the fundamental group of a non-contractible one-dimensional Peano continuum.
\end{cor}

\begin{customthm}{\ref{main2}}
Let $X$ be a topological space.  If $X$ is homotopy equivalent to spaces $Y_1$ and $Y_2$ where $Y_1$ is one-dimensional and $Y_2$ is a Peano continuum, then $X$ is homotopy equivalent to a one-dimensional Peano continuum.  If, in addition, $X$ is homotopy equivalent to a planar set, then $\pi_1(X,x_0)$ is isomorphic to the fundamental group of a one-dimensional planar Peano continuum.
\end{customthm}

\begin{proof}[Proof of Theorem \ref{main2}]  Let $X$ be a topological space and suppose that $X$ is homotopy equivalent to spaces $Y_1$ and $Y_2$ where $Y_1$ is one-dimensional and $Y_2$ is a Peano continuum.  By Proposition \ref{homotopic to image}, we have that $Y_2$ is homotopy equivalent to a one-dimensional Peano continuum contained in $Y_1$ and hence so is $X$.

Suppose, in addition, that $X$ is homotopy equivalent to a planar set $Y_3$.  Let $f_2: X \to Y_2$ be a homotopy equivalence with homotopy inverse $g_2: Y_2 \to X$.   Then $f_{2*}\bigl(\pi_1\bigl( g_2(Y_2), g_2(y_2)\bigr)\bigr)$ is isomorphic to $\pi_1(Y_2, Y_2)$ by Proposition \ref{homotopic to image}.  Since $f_{2*}$ is an isomorphism, we have that $\pi_1\bigl( g_2(Y_2), g_2(y_2)\bigr)$ is isomorphic to $\pi_1(X,x_0)$.  By Theorem \ref{main}, we have that $g_2(Y_2)$ is homotopy equivalent to a one-dimensional planar Peano continuum.
\end{proof}

\section{Example}\label{example}

\begin{prop}\label{not perfect}

The set of points at which a planar Peano continuum is not locally simply connected is not a perfect invariant of the homotopy type or of the fundamental group.

\end{prop}

\begin{proof}
  In \cite{ccz} Cannon, Conner, and Zastrow proved that a filled Sierpinski carpet is not homotopy equivalent to the standard Sierpinski carpet.  Specifically they showed that the filled Sierpinksi carpet has homotopy dimension two while the Sierpinski carpet has homotopy dimension one while the  set of points at which they are not locally simply connected is the same for both.  By applying Theorem \ref{main} to this example, we can see that the Sierpinksi carpet and the filled Sierpinksi carpet cannot have isomorphic fundamental groups.
\end{proof}

The fundamental group determines the set $B(X)$ with its topology, as well as the homotopy dimension for a planar Peano continuum $X$.  Thus a natural question is what other topologically defined properties of a planar continuum are determined by the fundamental group.

The following is an example of two planar Peano continua with the same homotopy dimensions for which it is unknown if they have isomorphic fundamental groups or if they are homotopy equivalent.

\begin{example}
A Warsaw circle is a space homeomorphic to

$$ \Bigl\{ \bigl(x, \sin(\pi/x)\bigr) \mid 0< x\leq 2\Bigr\}\cup \Bigl\{(x,y) \mid x\in\{0,2\}, \hspace{.5em} y\in [-2,1]\Bigr\}\cup \Bigl\{(x,-2) \mid x\in [0,2]\Bigr\}.$$

Notice that every Warsaw circle is tamely embedded into $\mathbb S^2$, i.e. the complement has exactly two simply connected components.  Let $W$ be a Warsaw circle  in $\mathbb S^2$ with  complementary components $C_1,C_2$.  Let $D(Y_i)$ be a null sequence of open discs in $C_i$ with limit set $W$.  Let $X= \mathbb S^2\backslash D(Y_1)$ and $Y = \mathbb S^2\backslash D(Y_2)$.

\end{example}

Since there is no continuous map fixing $W$ which interchanges the complementary components of $W$ in $\mathbb S^2$, it is not clear if $X$ and $Y$ should be homotopy equivalent or not.

\begin{quest}
Are $X$ and $Y$ homotopy equivalent?  Or possible weaker, do they have isomorphic fundamental groups?
\end{quest}

\begin{quest}
If $X$ and $Y$ are not homotopy equivalent, is there an invariant that can be used to distinguish between these two spaces?
\end{quest}

\def\cprime{$'$} \def\cprime{$'$}


\end{document}